\documentclass[12pt, reqno]{amsart}
\usepackage{amsmath, amsthm, amscd, amsfonts, amssymb, graphicx, color}
\usepackage[bookmarksnumbered, colorlinks, plainpages]{hyperref}
\hypersetup{colorlinks=true,linkcolor=red, anchorcolor=green, citecolor=cyan, urlcolor=red, filecolor=magenta, pdftoolbar=true}

\newtheorem{theorem}{Theorem}[section]
\newtheorem{lemma}[theorem]{Lemma}

\theoremstyle{definition}
\newtheorem{definition}[theorem]{Definition}
\newtheorem{example}[theorem]{Example}

\theoremstyle{remark}
\newtheorem{remark}[theorem]{Remark}
\numberwithin{equation}{section}

\begin{document}

\setcounter{page}{1}

\title[Banach's Principle in
Cone Modular Spaces]{Banach Contraction Principle in
Cone Modular Spaces \\ with Banach Algebra}

\author[M.\,\"{O}zav\c{s}ar \MakeLowercase{and} H. \c{C}ay]{Muttalip \"{O}zav\c{s}ar$^1$ \MakeLowercase{and} Hatice \c{C}ay$^2$}

\address{$^{1}$Department of Mathematics, Yildiz Technical University, Istanbul, Turkey.}
\email{\textcolor[rgb]{0.00,0.00,0.84}{mozavsar@yildiz.edu.tr}}

\address{$^{2}$IMU Vocational School, Istanbul Medipol University, Istanbul, Turkey.}
\email{\textcolor[rgb]{0.00,0.00,0.84}{hcay@medipol.edu.tr}}




\date{Received: xxxxxx; Revised: yyyyyy; Accepted: zzzzzz.
\newline \indent $^{*}$Corresponding author}

\begin{abstract}
Our aim in this paper is to present a new type of the modular space. This space contains the classical modular space. There are some mappings that do not have contractive condition in the usual modular space but become contraction in this new space.
\end{abstract} \maketitle

\section{\textbf{Introduction}}
In 1922, Banach presented a fixed point theorem known as Banach Contraction Principle (BCP) that is one of the important mathematical tools in nonlinear analysis. Then many authors dealth with this theorem in different spaces. For example, in 2014, Ma et al. \cite{Ref6} presented this theorem in $C^{\ast }$-algebra-valued metric space and claimed that this is a generalization of BCP in the standart metric space. But later, in 2016 Alsulami et al. \cite{Ref7}, Kadelburg and Radenovi\'{c} \cite{Ref8} separately showed that BCP obtained in $C^{\ast }$-algebra-valued metric space is equivalent to the result of BCP in the classical metric space.

In 1950, Nakano introduced the notion of modular space \cite{RefNA}. Then Musielak and Orlicz \cite{Ref5}, \cite{RefMUS} generalized the modular space. By using the results of \cite{RefMUS}, \cite{Ref5} Khamsi et al. \cite{Ref3} extended BCP to the frame of modular function space, an example of modular space, introduced by Kozlowski \cite{Ref2}. Inspired by the notion of $C^{\ast }$-algebra-valued metric space \cite{Ref6}, Shateri \cite{Ref1} presented a generalization for modular space.

Now in this work motivated by \cite{Ref7} and \cite{Ref8}, we firstly show that BCP in the setting of $C^{\ast }$-algebra-valued modular space does not provide a real extension for the BCP in the modular space \cite{Ref1}. Secondly, we introduce a new setting, namely, a cone modular space over Banach algebra, which enables us to obtain a proper generalization for BCP in the usual modular spaces. Finally, we conclude our work with an example. 

\subsection{Preliminaries}
Modular functional is defined as follows: \\
Let $\mathcal{V}$ be a vector space and $\varrho :\mathcal{V} \rightarrow \left[
0,\infty \right] $ be a functional for $x,y\in \mathcal{V}$. $\theta_{\mathcal{V}}$ represents the zero vector of $\mathcal{V}$. $\varrho $ is called modular if the followings hold:

m1.) $\varrho \left( x\right) =0$ if and only if $x=\theta_{\mathcal{V}}$.

m2.) $\varrho \left( \mu x\right) =\varrho \left( x\right) $ for each scalar with 
$\left\vert \mu \right\vert =1$.

m3.) $\varrho \left( \mu x+\alpha y\right) \leq \varrho \left( x\right) +\varrho
\left( y\right) $ if $\mu =1- \alpha$ for $\mu ,\alpha \geq 0.$

It is clear that the set $$\mathcal{V}_{\varrho }=\left\{ x\in \mathcal{V}:\varrho \left( \lambda x\right) \rightarrow 0 \
 \text{as} \ \lambda \rightarrow 0\right\}$$ is a vector subspace of $\mathcal{V}$. $\mathcal{V}_{\varrho }$ is called modular space.  

In addition to the conditions above, if $\varrho \left( \mu x+\alpha y\right) \leq \mu \varrho \left( x\right)
+\alpha \varrho \left( y\right) $ for $\mu ,\alpha \geq 0,\mu =1- \alpha$, then the functional $\varrho$ is called convex.
\begin{definition}
The modular $\varrho $ satisfies the $\Delta _{2}$-condition if $\lim_{n\rightarrow \infty }\varrho \left( 2x_{n}\right) =\theta _{\mathcal{V}}$  whenever  $\lim_{n\rightarrow \infty }\varrho \left( x_{n}\right) =\theta _{\mathcal{V}}$.
\end{definition}
We see from \cite{Ref3} that the BCP is valid for a mapping $T:M\rightarrow M$  where $M$ is a closed, bounded non-empty subset of modular function space:
\begin{theorem} \label{khamsi}
Let $p$ be a modular functional that satisfies the $\Delta_{2}$-condition and $M$ be a non-empty $p$-closed subset of the modular function space $\mathcal{V}_{p}$. If $T:M\rightarrow M$ is Lipschitzian and $M$ is $p$-bounded, then T has a unique fixed point.
\end{theorem} 

Now before giving our first result let recall some basic definitions and results from \cite{Ref9} and \cite{Ref6}.  
An algebra is unital if it has the multiplicative unit. An involution on a unital algebra $\mathcal{C}$ is a conjugate-linear map $a\rightarrow a^{\ast }$ on $\mathcal{C}$ such that $a^{\ast \ast }=a$ and $\left( ab\right) ^{\ast
}=b^{\ast }a^{\ast }$ for all $a,b\in \mathcal{C}.$ $\left(\mathcal{C},\ast\right)$ is said to be a $\ast$-algebra. A Banach $\ast$-algebra is a $\ast$-algebra with a complete submultiplicative norm such that $\left\Vert a^{\ast }\right\Vert_{\mathcal{C}}
=\left\Vert a\right\Vert_{\mathcal{C}}$ for each element $a$ of it. A $C^{\ast }$-algebra is a Banach $\ast $-algebra such that $\left\Vert a^{\ast }a\right\Vert_{\mathcal{C}}
=\left\Vert a\right\Vert_{\mathcal{C}} ^{2}$ for every element $a$ of it. In the rest we suppose that $\mathcal{C}$ is a unital $C^{\ast }$-algebra. $\sigma\left( x\right)$ stands for the spectrum of $x$. $\theta_{\mathcal{C}}$ represents the zero element of $\mathcal{C}$. The set $\mathcal{C}^{\#}=\{x\in \mathcal{C}:x^{\ast }=x\}$ denotes the hermitian or self-adjoint elements of $\mathcal{C}$. If $x\in \mathcal{C}^{\#}$ and $\sigma \left( x\right) \subset \left[ 0,\infty \right) $, then $x\in \mathcal{C}$ is said to be a positive element of $\mathcal{C}$. $\mathcal{C}^{+}$ denotes the positive elements of $\mathcal{C}$ and $\left\vert x\right\vert =\left( x^{\ast }x\right) ^{\frac{1}{2}}$. Thus a partial ordering $\preceq $ on $\mathcal{C}^{\#}$
is defined as $x\preceq y$ iff $y-x\in \mathcal{C}^{+}.$ 
Now let's recall the following theorem that will be used later

\begin{theorem} \label{theor1}
The following conditions hold for $\mathcal{C}$:

i) There is a unique element $b\in \mathcal{C}^{+}$ such that $b^{2}=a$ for $a\in \mathcal{C}^{+}$.

ii) The set $\mathcal{C}^{+}$ is equal to $\left\{ a^{\ast }a:a\in \mathcal{C}\right\} .$

iii) If $a,b\in\mathcal{C}^{\#}$ and $\theta_{\mathcal{C}}\preceq a\preceq b,$ then $\left\Vert
a\right\Vert_{\mathcal{C}} \preceq \left\Vert b\right\Vert_{\mathcal{C}} .$

iv) If $a,b\in\mathcal{C}^{\#}$, $c\in \mathcal{C}$ and $a\preceq b$, then $c^{\ast }ac\preceq c^{\ast }bc$.
\end{theorem}

In \cite{Ref6}, Ma et al. introduced the notion of $C^{\ast }$-algebra-valued metric space and proved BCP in such spaces. Then motivated by the results obtained in \cite{Ref6}, Shateri presented the notion of $C^{\ast }$-algebra-valued modular space in \cite{Ref1} as follows:

\begin{definition}
Let $\mathcal{V}$ be a vector space over $\mathbf{K}$. The functional $\rho :\mathcal{V}\rightarrow \mathcal{C}$ called $C^{\ast }$-algebra-valued modular if the followings hold: 

cm1) $\rho \left( x\right) \succeq \theta _{\mathcal{C}}$ and $\rho \left( x\right)
=\theta _{\mathcal{C}}$ if and only if $x=\theta_{\mathcal{V}}$.

cm2) $\rho \left( \alpha x\right) =\rho \left( x\right) $ for each $\alpha
\in \mathbf{K}$ with $\left\vert \alpha \right\vert =1$.

cm3) $\rho \left( \alpha x+\beta y\right) \preceq \rho \left( x\right) +\rho
\left( y\right) $ if $\alpha ,\beta \geq 0$ and $\alpha =1-\beta ,$ for arbitrary $x,y\in \mathcal{V}.$
\end{definition}
\noindent Note that the subset
 $$\mathcal{V}_{\rho }=\left\{ x\in \mathcal{V}:\lim_{\lambda \rightarrow 0}\rho \left( \lambda
x\right) =\theta _{\mathcal{C}}\right\}$$ is a subspace of $\mathcal{V}$, and $\mathcal{V}_{\rho }$ is called $C^{\ast }$-algebra-valued modular space.

\begin{definition} \label{def3}
Let $\mathcal{V}_{\rho }$ be a $C^{\ast }$-algebra-valued modular space. Then a mapping $T:\mathcal{V}_{\rho }\rightarrow \mathcal{V}_{\rho }$ is called a $C^{\ast }$-algebra-valued contractive mapping on $\mathcal{V}_{\rho }$ if there is $k\in \mathcal{C}$ with $\Vert k\Vert <1$ and $\alpha
,\beta \in 
\mathbb{R}
^{+}$ with $\alpha >\beta $ such that 
$$\rho \left( \alpha \left( Tx-Ty\right) \right) \preceq k^*\rho \left( \beta
\left( x-y\right) \right )k$$
for all $x,y\in \mathcal{V}$.
\end{definition}

In \cite{Ref1}, Shateri gives definitions of $\rho $-convergence, $\Delta_2$-condition, $\rho $-Cauchy and $\rho $-completeness in accordance with literature and introduces the following theorem:

\begin{theorem} \label{shateri}
Suppose that $\mathcal{V}_\rho $ is a $\rho $-complete modular space with the $\Delta_2$-condition and $T$
is a $C^{\ast }$-algebra-valued contractive mapping on $\mathcal{V}_\rho .$ Then $T$
has a unique fixed point in $\mathcal{V}_{\rho }.$
\end{theorem} Now we recall the definition of the Banach algebra and some associated properties from \cite{Ref13,Ref14} that will be necessary for our results.
\begin{definition} Let $\mathcal{A}$ be a Banach space over $\mathbf{K}\in \left\{\mathbb{R},\mathbb{C}\right\}$ and $\left\Vert . \right\Vert_{\mathcal{A}} $ be a norm on $\mathcal{A}$. $\mathcal{A}$ is said to be a Banach algebra if there is an
operation of multiplication satisfying the following conditions:

\begin{enumerate}
\item $\left( u+v\right) w=uw+vw$ and $u\left( v+w\right) =uv+uw$.

\item $\left( uv\right) w=u\left( vw\right) $.

\item $\beta \left( uv\right) =\left( \beta u\right) v=u\left( \beta
v\right) $.

\item $\left\Vert uv\right\Vert_{\mathcal{A}} \leq \left\Vert u\right\Vert_{\mathcal{A}} \left\Vert
v\right\Vert_{\mathcal{A}} $.
\end{enumerate}
for all $u,v,w\in \mathcal{A}$ and $
\beta \in \mathbf{K} $. If there is an element $e\in  \mathcal{A}$ such that $ae=ea=a$ for all $a\in \mathcal{A}$, then $e$ is called the multiplicative unit of the Banach algebra $\mathcal{A}$. An element $a\in \mathcal{A}$ is called invertible if there is $a^{-1}\in \mathcal{A}$ such that $aa^{-1} =a^{-1}a=e$. From now on, we suppose that $\mathcal{A}$ is a Banach algebra with the multiplicative unit $e$ and zero vector $\theta_{\mathcal{A}}$.
\end{definition}

\begin{definition} Let $P\subset \mathcal{A}$, then $P$ is called cone if the followings hold:

\begin{enumerate}
\item $\left\{ e,\theta_{\mathcal{A}} \right\} \subset P$.

\item $\mu P+\beta P\subset P$ where all $\mu ,\beta $ are non-negative real
numbers.

\item $PP=P^{2}\subset P$.

\item $P\cap \left( -P\right) =\left\{ \theta_{\mathcal{A}} \right\} .$
\end{enumerate}
\end{definition}
A partial ordering $\preceq $ on $\mathcal{A}$ is defined as $u\preceq v$ if and only if $v-u\in
P$. $u\prec v$ stands for $u\preceq v$ and $u\neq v.$ $intP$ denotes the
interior of $P$. $u\ll v$ represents $v-u\in intP$. $P$ is said to be a
solid cone if $intP\neq \emptyset $.  The cone $P$ is said to be normal if there exists $L>0$ such that for all $x,y\in \mathcal{A}$, $\theta_{\mathcal{A}}\preceq x\preceq y$ implies $\Vert x\Vert_{\mathcal{A}}\leq L \Vert y\Vert_{\mathcal{A}}$. From now on, $P$ denotes a normal solid cone of $\mathcal{A}$ unless otherwise stated.
\begin{definition} Let $X$ be a nonempty set and  $d:X\times X\rightarrow \mathcal{A}$ be a mapping holding the following conditions: 
\begin{enumerate}
\item $\theta_{\mathcal{A}} \preceq d\left( u,v\right) $ for all $u,v\in X$ and $d\left(
u,v\right) =\theta_{\mathcal{A}} $ if and only if $u=v$.

\item $d\left( u,v\right) =d\left( v,u\right) $ for all $u,v\in X$.

\item $d\left( u,w \right) \preceq d\left( u,v\right) +d\left( v,w
\right) $ for all $u,v,w \in X$.
\end{enumerate}
Then $\left( X,d\right) $ is said to be a cone metric space over $\mathcal{A}$.
\end{definition}

BCP in such spaces is introduced by Liu and Xu \cite{Ref14} as follows:
\begin{theorem} Let $\left( X,d\right) $ be a cone metric space over $\mathcal{A}$ and $P$ be a normal solid cone of $\mathcal{A}$ where $a\in P$ with $r\left(
a\right)<1$. If the mapping $T:X\rightarrow X$ holds following condition for all $x,y \in X$, then it has a unique fixed point in $X$:
$$ d\left( Tx,Ty\right)\preceq ad\left( x, y\right).$$
\end{theorem}

After the announcement of this theorem, in \cite{Ref15}, Xu and Radenovi\'{c} showed that there is no need to normality condition to prove BCP mentioned above. However, we must note that as a generalization of the usual modular space, a cone modular space in this paper can be defined if $P$ holds the normality condition. 

\begin{lemma} \label{lem3}
The spectral radius $r\left( a\right) $ of $a\in \mathcal{A}$ holds
$$r\left( a\right) =\lim_{n\rightarrow \infty }\left\Vert a^{n}\right\Vert_{\mathcal{A}} ^{
\frac{1}{n}}=\inf \left\Vert a^{n}\right\Vert_{\mathcal{A}} ^{\frac{1}{n}}.$$
If $r\left( a\right) <1,$ then $e-a$ is invertible in $\mathcal{A}.$ Furthermore
$$\left( e-a\right)^{-1}=\sum\limits_{i=0}^{\infty }a^{i}.$$
\end{lemma}

\section{\textbf{Main results}}
In the sequel we first show that BCP in $C^{\ast }$-algebra-valued modular spaces (see \cite{Ref1}) is equivalent to BCP in the usual modular spaces:

\begin{theorem}
BCP in the sense of Theorem \ref{shateri} is equivalent to one in the usual modular space.
\end{theorem}

\begin{proof}
From the Definition \ref{def3}, we know that there is $a\in \mathcal{C}$ with $\Vert a \Vert_{\mathcal{C}}<1$ and $\alpha
,\beta \in  
\mathbb{R}
^{+}$ with $\alpha >\beta $ such that  $\rho \left( \alpha \left( Tx-Ty\right) \right) \preceq a^{\ast }\rho
\left( \beta \left( x-y\right) \right) a$ for all $x,y\in \mathcal{V}$. Moreover, by (ii) in Theorem
\ref{theor1}, we see that there exists $u_{f}\in \mathcal{C}$ such
that $\rho \left( \beta \left( x-y\right) \right) =u_{f}^{\ast }u_{f}$. Hence $\left\Vert \rho \left( \beta \left( x-y\right) \right)
\right\Vert_{\mathcal{C}} =\left\Vert u_{f}^{\ast }u_{f}\right\Vert_{\mathcal{C}} =\left\Vert
u_{f}\right\Vert_{\mathcal{C}} ^{2}.$ 
On the other hand, since
\begin{equation*}
\rho \left( \alpha \left( Tx-Ty\right) \right) \preceq a^{\ast }\rho \left(
\beta \left( x-y\right) \right) a=a^{\ast }u_{f}^{\ast }u_{f}a=\left(
u_{f}a\right) ^{\ast }u_{f}a,
\end{equation*}
then by using (iii) in Theorem \ref{theor1} we obtain
\begin{flalign} \label{equ2.1}
\begin{aligned}
\left\Vert \rho \left( \alpha \left( Tx-Ty\right) \right) \right\Vert_{\mathcal{C}}
\preceq \left\Vert \left( u_{f}a\right) ^{\ast }u_{f}a\right\Vert_{\mathcal{C}}
=\left\Vert u_{f}a\right\Vert_{\mathcal{C}}^{2} & \\ \preceq \left\Vert a\right\Vert_{\mathcal{C}}
^{2}\left\Vert u_{f}\right\Vert_{\mathcal{C}}^{2}=\left\Vert a\right\Vert_{\mathcal{C}} ^{2}\left\Vert
\rho \left( \beta \left( x-y\right) \right) \right\Vert_{\mathcal{C}}. &
\end{aligned}
\end{flalign}
Now consider a mapping $F:\mathcal{V}_{\varrho }\rightarrow \left[
0,\infty \right] $ such as $F(x)=\Vert \rho(x)\Vert_{\mathcal{C}}$. Then $F$ is a usual modular. Indeed, 
\begin{itemize}
\item[i)] Let $F(x)=0$. Then $\Vert \rho(x)\Vert_{\mathcal{C}}=0$. Thus, by the property of norm, we get $\rho(x)=0$. Since $\rho$ is a modular, then we have $x=\theta _{\mathcal{V}}$.
\item[ii)] Let $\mu$ be a scalar with $\left\vert \mu \right\vert=1$. Then $F(\mu x)=\Vert \rho(\mu x)\Vert_{\mathcal{C}}=\Vert \rho(x)\Vert_{\mathcal{C}}=F(x)$.
\item[iii)] Let $\mu =1- \lambda$ for $\mu ,\lambda \geq 0$. Then, by using (iii) in Theorem \ref{theor1} and triangle inequality of the norm, we obtain $$F(\mu x+\lambda y)=\Vert \rho(\mu x+\lambda y)\Vert_{\mathcal{C}}\leq \Vert \rho(x)+\rho(y)\Vert_{\mathcal{C}}\leq \Vert \rho(x) \Vert_{\mathcal{C}}+\Vert \rho(y) \Vert_{\mathcal{C}}=F(x)+F(y).$$
\end{itemize}
By letting $k=\Vert a\Vert_{\mathcal{C}}^{2}$, we see that $k<1$. Thus, by (\ref{equ2.1}), we obtain
$$ F \left( \alpha \left( Tx-Ty\right) \right) \preceq k
F \left( \beta \left( x-y\right) \right).$$
Hence, BCP in $C^{*}$-algebra valued modular spaces is equivalent to one in the usual modular spaces.
\end{proof}

Now we introduce a proper space where we introduce a proper generalization for BCP in the usual modular space.
\begin{definition} \label{def1}
Let $\mathcal{V}$ be a vector space over $\mathbf{K}$. A mapping $\rho :\mathcal{V}\rightarrow \mathcal{A}$ is called a cone modular functional if it satisfies the followings:
\begin{itemize}
\item[cmf1] $\rho \left( u\right) \succeq \theta _{\mathcal{A}}$ and $\rho \left( u\right)
=\theta _{\mathcal{A}}$ if and only if $u=\theta _{\mathcal{V}}$.

\item[cmf2] $\rho \left( \alpha u\right) =\rho \left( u\right) $ for each $\alpha
\in \mathbf{K}$ with $\left\vert \alpha \right\vert =1$.

\item[cmf3] \label{cmf3} $\rho \left( \alpha u+\beta v\right) \preceq \rho \left( u\right) +\rho
\left( v\right) $ if $\alpha ,\beta \geq 0$ and $\alpha +\beta =1$.
\end{itemize}
for all $u,v\in \mathcal{V}.$ In addition to the conditions above, if $\rho$ satisfies $\rho \left( \alpha
x+\beta y\right) \preceq \alpha \rho \left( x\right) +\beta \rho \left(
y\right) $ whenever $\alpha ,\beta \geq 0$ and $\alpha =1-\beta$, then $\rho$ is called convex.
\end{definition} 
Now we need to point out that
$$\mathcal{V}_{\rho }=\left\{ x\in \mathcal{V}:\lim_{\lambda \rightarrow 0}\rho \left( \lambda
x\right) =\theta _{\mathcal{A}}\right\}$$ is a subspace of $\mathcal{V}$. Indeed,
\begin{itemize}
\item[i)] Let $x,y\in \mathcal{V}_{\rho }$. Then $\lim_{\lambda \rightarrow 0}\rho \left( \lambda
x\right) =\theta _{\mathcal{A}}$ and $\lim_{\lambda \rightarrow 0}\rho \left( \lambda
y\right) =\theta _{\mathcal{A}}$. By using (cmf3), $\rho\left( \lambda(x+y)\right) =\rho\left( \frac{1}{2}\left( 2\lambda x+2\lambda y\right) \right) \preceq
\rho\left( 2\lambda x\right) +\rho\left( 2\lambda y\right).$ Taking $t=2\lambda$, we see that $t\rightarrow 0$ as $\lambda \rightarrow 0$. So we get $\theta_{\mathcal{A}}\preceq \lim_{\lambda \rightarrow 0}\rho \left( \lambda
\left( x+y\right) \right)\preceq \theta _{\mathcal{A}}$. Thus, by the normality of cone, we can use Sandwich Theorem. Therefore, we have $\lim_{\lambda \rightarrow 0}\rho \left( \lambda
\left( x+y\right) \right)= \theta _{\mathcal{A}}$, implying that $x+y\in \mathcal{V}_{\rho }$.
\item[ii)] Take an arbitrary $\alpha \in \mathbf{K}$ and $x\in \mathcal{V}_{\rho }$. Then $\lim_{\lambda \rightarrow 0}\rho \left( \lambda
x\right) =\theta _{\mathcal{A}}$. Letting $\alpha \lambda = t$ we have $t\rightarrow 0$ as $\lambda \rightarrow 0$. Hence $\lim_{\lambda \rightarrow 0}\rho \left( \lambda
 \alpha x\right) =\theta _{\mathcal{A}}$. So $\alpha x \in \mathcal{V}_{\rho }$.
\end{itemize}
From now on, we call $\mathcal{V}_{\rho }$ a cone modular space over Banach algebra $\mathcal{A}$. Note that the cone modular space over $\mathcal{A}$ is a generalization of the usual modular space. In the sequel, we introduce some basic definitions.
Let us define a functional on $\mathcal{V}_{\rho }$ such that $\Vert x\Vert_{F}=inf\lbrace\delta>0:\Vert\rho(\frac{x}{\delta })\Vert_{\mathcal{A}}\leq\delta\rbrace$. Note that $\Vert.\Vert_{F}$ is an F-norm, that is, it satisfies the following conditions:

i) $\left\Vert x\right\Vert_{F} =0$ if and only if $x=\theta _{V}$.

ii) $\left\Vert x+y\right\Vert_{F} \leq \left\Vert x\right\Vert_{F} +\left\Vert
y\right\Vert_{F}$.

iii) $\left\Vert -x\right\Vert_{F} =\left\Vert x\right\Vert_{F}$.

iv) $\alpha _{n}\rightarrow \alpha $ and $\left\Vert x_{n}-x\right\Vert_{F}
\rightarrow 0$ imply $\left\Vert \alpha _{n}x-\alpha x\right\Vert_{F}
\rightarrow 0$.

\begin{definition} \label{defcomplet}
Let $\mathcal{V}_{\rho }$ be a cone modular space over $\mathcal{A}$ and $\left\{
x_{n}\right\} $ be in $\mathcal{V}_{\rho }.$ We say that

i) $\left\{ x_{n}\right\} $ is a $\rho $-convergent to $x\in \mathcal{V}_{\rho }$ denoted by $
x_{n}\rightarrow x$ $\left( n\rightarrow \infty \right) $ if for each $\varepsilon>0$ there is a natural number $N$ and $\mu >0$ such that $\Vert \rho \left(\mu\left( 
x_{n}-x\right) \right)\Vert_{\mathcal{A}}< \varepsilon $ for all $n\geq N$.

ii) $\left\{ x_{n}\right\} $ is a $\rho $-Cauchy if for each $\varepsilon>0$ there is a natural number $N$ and $\mu >0$ such that $\Vert\rho \left(\mu\left(
x_{n}-x_{m}\right)\right)\Vert_{\mathcal{A}}< \varepsilon$ for all $n,m\geq N$.

iii) $\mathcal{V}_{\rho }$ is $\rho $-complete if each $\rho $-Cauchy sequence with
respect to $\mathcal{A}$ is $\rho $-convergent.

iv) We say that $\rho$ satisfies $\Delta_2$-condition if for each $\varepsilon>0$ there is $n_0\in \mathbb{N}$ such that $\Vert\rho(2x_n)\Vert_{\mathcal{A}}< \varepsilon$ whenever $\Vert\rho(x_n)\Vert_{\mathcal{A}}< \varepsilon$ for $n\geq n_0$.
\end{definition}
\begin{remark}
Since $\Vert\rho\left(x\right) \Vert_{\mathcal{A}}\leq\Vert x\Vert_{F}$ for $\Vert x\Vert_{F}<1$, then the norm convergence implies modular convergence to the same limit.
\end{remark}
\begin{remark} \label{remark} 
If $0<\alpha <\beta ,$ then from Definition \ref{def1}, we have $\rho
\left( \alpha x\right) =\rho \left( \frac{\alpha }{\beta }\beta x\right)
\preceq \rho \left( \beta x\right) $ for all $x\in \mathcal{V}$ with $y=0.$
Furthermore, if $\rho $ is a convex cone modular on $\mathcal{V}$ and $\left\vert
\alpha \right\vert \leq 1,$ then $\rho \left( \alpha x\right) \preceq \alpha
\rho \left( x\right) $ for all $x\in \mathcal{V}.$
\end{remark}

In the sequel, we suppose that $\mathcal{V}_{\rho }$ is a cone modular space over Banach algebra $\mathcal{A}$. Now
we prove a fixed point theorem by introducing the notion of generalized contractive type mapping in the construction of cone modular spaces over Banach algebra.

\begin{definition} \label{mod}
A mapping $T:\mathcal{V}_{\rho }\rightarrow \mathcal{V}_{\rho }$ is called a cone contractive mapping on $\mathcal{V}_{\rho }$ if there exist a scalar vector $k\in P$ with $r\left( k\right) <1$ and $\alpha
,\beta \in 
\mathbb{R}
^{+}$ with $\alpha >\beta $ such that for all $x,y\in \mathcal{V}_{\rho }$
\begin{equation} \label{2.2}
\rho \left( \alpha \left( Tx-Ty\right) \right) \preceq k\rho \left( \beta
\left( x-y\right) \right).
\end{equation}
\end{definition}

\begin{theorem} \label{conethm}
Let $\mathcal{V}_{\rho }$ be a $\rho $-complete modular space with $\Delta_2$-condition and $T$ be a cone contractive mapping on $\mathcal{V}_{\rho }.$ Then $T$ has a
unique fixed point in $\mathcal{V}_{\rho }.$
\end{theorem}

\begin{proof}
If $k=\theta _{\mathcal{A}} $ the proof is clear. Thus, we assume that $k\neq \theta _{\mathcal{A}}.$ Let $\alpha _{0}\in 
\mathbb{R}
^{+}$ be with $\frac{\beta }{\alpha }+\frac{1}{\alpha _{0}}=1.$ For an
arbitrary $x\in V_{\rho }$ and $n\in 
\mathbb{N}
,$ set $x_{n+1}=Tx_{n}=T^{n+1}x.$ Since $\alpha>\beta$, then using Remark \ref{remark} and Definition \ref{mod}, we get
\begin{eqnarray*}
\rho \left( \beta \left( x_{n+1}-x_{n}\right) \right)  &=&\rho \left( \beta
\left( Tx_{n}-Tx_{n-1}\right) \right)  \\
&\preceq &\rho \left( \alpha \left( Tx_{n}-Tx_{n-1}\right) \right)  \\
&\preceq &k\rho \left( \beta \left( x_{n}-x_{n-1}\right) \right)  \\
&=&k\rho \left( \beta \left( Tx_{n-1}-Tx_{n-2}\right) \right)  \\
&\preceq &k\rho \left( \alpha \left( Tx_{n-1}-Tx_{n-2}\right) \right)  \\
&\preceq &k^{2}\rho \left( \beta \left( x_{n-1}-x_{n-2}\right) \right)  \\
&&. \\
&&. \\
&&. \\
&\preceq &k^{n}\rho \left( \beta \left( x_{1}-x_{0}\right) \right) .
\end{eqnarray*}
Since $\frac{\beta }{\alpha }+\frac{1}{\alpha _{0}}=1$, then using (cmf3) we have
\begin{eqnarray*}
\rho \left( \beta \left( x_{n+1}-x_{n-1}\right) \right)  &=&\rho \left(
\beta \left( x_{n+1}-x_{n}+x_{n}-x_{n-1}\right) \right)  \\
&=&\rho \left( \beta \left( x_{n+1}-x_{n})+\beta (x_{n}-x_{n-1}\right)
\right)  \\
&=&\rho \left( \beta \frac{\alpha }{\alpha }\left( x_{n+1}-x_{n})+\beta 
\frac{\alpha _{0}}{\alpha _{0}}(x_{n}-x_{n-1}\right) \right)  \\
&\preceq &\rho \left( \alpha \left( x_{n+1}-x_{n}\right) \right) +\rho
\left( \beta \alpha _{0}\left( x_{n}-x_{n-1}\right) \right) .
\end{eqnarray*}
Now since $\alpha> \beta$, by using (\ref{2.2}), we obtain from the inequality given above
\begin{eqnarray*}
\rho \left( \beta \left( x_{n+1}-x_{n-1}\right) \right)&\preceq &k\rho \left( \beta \left( x_{n}-x_{n-1}\right) \right) +\rho
\left( \beta \alpha _{0}\left( x_{n}-x_{n-1}\right) \right).  
\end{eqnarray*}
By applying recursively the approach used above, we get
\begin{eqnarray*}
\rho \left( \beta \left( x_{n+1}-x_{n-1}\right) \right)
&\preceq &k^{n}\rho \left( \beta \alpha _{0}\left( x_{1}-x_{0}\right)
\right) +k^{n-1}\rho \left( \beta \alpha _{0}\left( x_{1}-x_{0}\right)
\right).
\end{eqnarray*}
Thus, for $n+1>m,$ we obtain following inequality
\begin{eqnarray*}
\rho \left( \beta \left( x_{n+1}-x_{m}\right) \right)  &\preceq &\rho \left(
\alpha \left( x_{n+1}-x_{m+1}\right) \right) +\rho \left( \beta \alpha
_{0}\left( x_{m+1}-x_{m}\right) \right)  \\
&\preceq &\rho \left( \alpha \left( x_{n+1}-x_{m+1}\right) \right)
+k^{m}\rho \left( \beta \alpha _{0}\left( x_{1}-x_{0}\right) \right)  \\
&=&\rho \left( \alpha \left( T_{n}-T_{m}\right) \right) +k^{m}\rho \left(
\beta \alpha _{0}\left( x_{1}-x_{0}\right) \right)  \\
&\preceq &k\rho \left( \beta \left( x_{n}-x_{m}\right) \right) +k^{m}\rho
\left( \beta \alpha _{0}\left( x_{1}-x_{0}\right) \right)  \\
&\preceq &k\left[ \rho \left( \alpha \left( x_{n}-x_{m+1}\right) \right)
+\rho \left( \beta \alpha _{0}\left( x_{m+1}-x_{m}\right) \right) \right] \\ 
&+&k^{m}\rho \left( \beta \alpha _{0}\left( x_{1}-x_{0}\right) \right)  \\
&\preceq &k\rho \left( \alpha \left( x_{n}-x_{m+1}\right) \right) 
+k.k^{m}\rho \left( \beta \alpha _{0}\left( x_{1}-x_{0}\right) \right) \\
&+&k^{m}\rho \left( \beta \alpha _{0}\left( x_{1}-x_{0}\right) \right)  \\
&\preceq &k^{2}\rho \left( \beta \left( x_{n-1}-x_{m}\right) \right) \\
&+&\left\{ k^{m+1}+k^{m}\right\} \rho \left( \beta \alpha _{0}\left(
x_{1}-x_{0}\right) \right)  \\
&\preceq &k^{3}\rho \left( \beta \left( x_{n-2}-x_{m}\right) \right) \\
&+&\left\{ k^{m+2}+k^{m+1}+k^{m}\right\} \rho \left( \beta \alpha _{0}\left(
x_{1}-x_{0}\right) \right). 
\end{eqnarray*}
By induction, we obtain
\begin{eqnarray*}
\rho \left( \beta \left( x_{n+1}-x_{m}\right) \right)  &\preceq
&k^{n-m+1}\rho \left( \beta \left( x_{m}-x_{m}\right) \right) \\
&+&\left\{
k^{m+n-m}+...+k^{m+1}+k^{m}\right\} \rho \left( \beta \alpha
_{0}\left( x_{1}-x_{0}\right) \right)  \\
&=&k^{m}\left( e+k+k^{2}+...+k^{n-m}\right) \rho \left( \beta \alpha
_{0}\left( x_{1}-x_{0}\right) \right).
\end{eqnarray*}
Since $r\left( k\right) <1$, then by Lemma \ref{lem3}, we obtain that $e-k$ is
invertible and $\left( e-k\right) ^{-1}=\sum\limits_{i=0}^{\infty }k^{i}.$ Thus,
\begin{eqnarray*} 
\rho \left( \beta \left( x_{n+1}-x_{m}\right) \right)& 
\preceq &k^{m}\left[ \sum\limits_{i=0}^{\infty }k^{i}\right] \rho \left(
\beta \alpha _{0}\left( x_{1}-x_{0}\right) \right)  \\
&=&k^{m}\left( e-k\right) ^{-1}\rho \left( \beta \alpha _{0}\left(
x_{1}-x_{0}\right) \right) .
\end{eqnarray*}
Since $P$ is a normal solid cone with a normal constant $L$ and $\Vert k^m\Vert_{\mathcal{A}}\rightarrow0$ $\left( m\rightarrow\infty\right) $. Thus for $\left(m\rightarrow\infty\right)$ we get, 
\begin{eqnarray*} 
\Vert\rho \left( \beta \left( x_{n+1}-x_{m}\right) \right)\Vert_{\mathcal{A}}& 
\leq L \Vert k^{m}\Vert_{\mathcal{A}}\Vert\left( e-k\right)^{-1}\Vert_{\mathcal{A}}\Vert\rho \left( \beta \alpha _{0}\left(
x_{1}-x_{0}\right) \right)\Vert_{\mathcal{A}}\rightarrow 0.
\end{eqnarray*}
Thus $\left\{ x_{n}\right\} $ is a $\rho $-Cauchy sequence. 
Since $V_{\rho }$ is a $\rho $-complete cone modular space over Banach algebra $\mathcal{A},$
there exist $x^*\in V_{\rho }$ and $\alpha>0$ such that $$\Vert\rho \left(\alpha\left(  x_{n}-x^* \right)\right)\Vert_{\mathcal{A}}=\Vert\rho
\left(\alpha\left(  Tx_{n-1}-x^*\right)\right)\Vert_{\mathcal{A}}<c.$$ 
Now, it remains to show that $x^*$ is a fixed point of $T.$ Indeed,
\begin{eqnarray*}
\rho \left( \frac{\alpha }{2}\left( Tx^*-x^*\right) \right)  &=&\rho \left( 
\frac{\alpha }{2}\left( Tx^*-T^{n+1}x\right) +\frac{\alpha }{2}\left(
T^{n+1}x-x^*\right) \right)  \\
&\preceq &\rho \left( \alpha \left( Tx^*-T^{n+1}x\right) \right) +\rho \left(
\alpha \left( T^{n+1}x-x^*\right) \right)  \\
&\preceq &k\rho \left( \beta \left( x^*-T^{n}x\right) \right) +\rho \left(
\alpha \left( T^{n+1}x-x^*\right) \right)  \\
&\preceq &k\rho \left( \alpha \left( x^*-T^{n}x\right) \right) +\rho \left(
\alpha \left( T^{n+1}x-x^*\right) \right) .
\end{eqnarray*} 
So,
\begin{align*}
\Vert\rho \left( \frac{\alpha }{2}\left( Tx^*-x^*\right) \right)\Vert_{\mathcal{A}}\leq L\left(\Vert k\Vert_{\mathcal{A}}\Vert\rho \left( \alpha \left( x^*-T^{n}x\right) \right)\Vert_{\mathcal{A}} +\Vert\rho \left(
\alpha \left( T^{n+1}x-x^*\right) \right)\Vert_{\mathcal{A}}\right).
\end{align*}
For $(n\rightarrow\infty)$, $L\left(\Vert k\Vert_{\mathcal{A}}\Vert\rho \left( \alpha \left( x^*-T^{n}x\right) \right)\Vert_{\mathcal{A}} +\Vert\rho \left(
\alpha \left( T^{n+1}x-x^*\right) \right)\Vert_{\mathcal{A}}\right)\rightarrow 0$. Thus we have $\Vert\rho \left( \frac{\alpha }{2}\left( Tx^*-x^*\right) \right)\Vert_{\mathcal{A}}
=0.$ Therefore $Tx^*=x^*.$ Now assume that $y^*\left( \neq x^*\right) $
be another fixed point of $T.$ Then we get
\begin{eqnarray*}
\rho \left( \beta \left( x^*-y^*\right) \right)  &=&\rho \left( \beta \left(
Tx^*-Ty^*\right) \right)  \\
&\preceq &\rho \left( \alpha \left( Tx^*-Ty^*\right) \right)  \\
&\preceq &k\rho \left( \beta \left( x^*-y^*\right) \right)  \\
&\preceq &k^{2}\rho \left( \beta \left( x^*-y^*\right) \right)  \\
&&. \\
&&. \\
&\preceq &k^{n}\rho \left( \beta \left( x^*-y^*\right) \right) .
\end{eqnarray*}
Since
\begin{eqnarray*}
\Vert\rho \left( \beta \left( x^*-y^*\right) \right)\Vert_{\mathcal{A}}\leq L\Vert k^{n}\Vert_{\mathcal{A}}\Vert\rho \left( \beta \left( x^*-y^*\right) \right)\Vert_{\mathcal{A}}\rightarrow 0
\end{eqnarray*}
while $n\rightarrow\infty$, then we have $\rho \left( \beta \left( x^*-y^*\right) \right) =\theta _{\mathcal{A}}$ and so $x^*=y^*.$ Hence the fixed point is unique.
\end{proof}

Now we present an example to show that our result provides a real generalization of the fixed point theory in the modular spaces: 

\begin{example}
Let $\mathcal{A}=
\mathbb{R}
^{2}.$ For each $\left( b_{1},b_{2}\right) \in \mathcal{A},$ $\left\Vert \left(
b_{1},b_{2}\right) \right\Vert_{\mathcal{A}} =\left\vert b_{1}\right\vert +\left\vert
b_{2}\right\vert .$ The multiplication is defined as $ba=\left(
b_{1},b_{2}\right) \left( a_{1},a_{2}\right) =\left(
b_{1}a_{1},b_{1}a_{2}+b_{2}a_{1}\right) .$ Then it is obvious that $\mathcal{A}$ is a
Banach algebra with unit $e=\left( 1,0\right).$ Let 
\begin{equation*}
P=\left\{ \left(
b_{1},b_{2}\right) \in 
\mathbb{R}
^{2}:b_{1},b_{2}\geq 0\right\}.
\end{equation*}
Thus $P$ is a normal solid cone with a constant $L=1$. Let $V=\mathbb{R}^{2}$ and the cone modular $\rho $ be defined by $\rho \left(
b\right) =\rho \left( \left( b_{1},b_{2}\right) \right) =\left( \left\vert
b_{1}\right\vert ,\left\vert b_{2}\right\vert \right).$ So, $\rho \left(
b\right) \in P.$ Then $V_{\rho }=\left\{ b\in
V:\lim_{\lambda \rightarrow 0}\rho \left( \lambda b\right) =\theta
_{\mathcal{A}}\right\} $ is a $\rho $-complete cone modular space over $\mathcal{A}.$ We define the
mapping $T:V_{\rho }\rightarrow V_{\rho }$ by 
\begin{equation*}
T\left( b\right) =T\left( \left( b_{1},b_{2}\right) \right) =\left( \log
\left( 4+\left\vert b_{1}\right\vert \right) ,\arctan \left( 3+\left\vert
b_{2}\right\vert \right) +\lambda b_{1}\right) ,
\end{equation*}
where $\lambda $ can be any large positive real number. By Lagrange mean value theorem we get
\begin{eqnarray*}
\rho \left( \alpha \left( T\left( b_{1},b_{2}\right) -T\left(
a_{1},a_{2}\right) \right) \right)  &\preceq &\left( \frac{\alpha }{4}
\left\vert b_{1}-a_{1}\right\vert ,\frac{\alpha }{10}\left\vert
b_{2}-a_{2}\right\vert +\lambda \left( b_{1}-a_{1}\right) \right)  \\
&\preceq &\left( \frac{1}{2},\lambda \right) \rho \left( \frac{\alpha }{2}
\left( \left( b_{1},b_{2}\right) -\left( a_{1},a_{2}\right) \right) \right).
\end{eqnarray*}
Since $r\left( \left( \frac{1}{2},\lambda \right) \right)
=\lim_{n\rightarrow \infty }\left\Vert \left( \frac{1}{2},\lambda \right)
^{n}\right\Vert_{\mathcal{A}} ^{\frac{1}{n}}=\frac{1}{2}<1$, then by Theorem \ref{conethm}, $T$ has a
unique fixed point theorem in $\mathcal{A}$. Now we show that $T$ is not a contraction in the setting of usual modular spaces. Indeed, we first let
 $\rho ^{\ast }=\xi _{c}\circ \rho $ where $c\in intP$ and $\xi _{c}:\mathcal{A}\rightarrow 
\mathbb{R}$ is the nonlinear scalarization
function defined by $\xi _{c}\left( b\right) =\inf
\left\{ t\in 
\mathbb{R}
:b\in tc-P\right\}=\inf\lbrace t\in \mathbb{R} : b\leq tc\rbrace$ (see \cite{Ref19}). Therefore, since $intP=\left\{ \left( c_{1},c_{2}\right) \in 
\mathbb{R}
^{2}:c_{1},c_{2}> 0\right\}$, we have  
\begin{equation*}
\xi _{c}\left( b\right) =\xi _{c}\left( \left( b_{1},b_{2}\right) \right)
=\inf \left\{ t\in 
\mathbb{R}
:\left( b_{1},b_{2}\right) \leq t\left( c_{1},c_{2}\right) \right\} =\max
\left\{ \frac{b_{1}}{c_{1}},\frac{b_{2}}{c_{2}}\right\} 
\end{equation*}
for $c=\left(
c_{1},c_{2}\right) \in intP$ and $b=\left( b_{1},b_{2}\right) \in \mathcal{A}$. Thus,
\begin{equation*}
\rho ^{\ast }\left( a\right) =\left( \xi _{c}\circ
\rho \right) \left( a_{1},a_{2}\right) =\max \left\{ \frac{\left\vert
a_{1}\right\vert }{c_{1}},\frac{\left\vert a_{2}\right\vert }{
c_{2}}\right\}
\end{equation*}
for $a,b\in V$.

 Now let $\alpha \succ \frac{c_{2}}{c_{1}}$ and consider $
a=\left( 1,0\right) ,b=\left( 0,0\right) .$ We have
\begin{equation*}
\rho ^{\ast }\left( Ta-Tb\right) =\max \left\{ \frac{\log 5-\log 4}{c_{1}},
\frac{\alpha }{c_{2}}\right\} \succeq \frac{\alpha }{c_{2}}\succ \frac{1}{
c_{1}}=\rho ^{\ast }\left( a-b\right),
\end{equation*}
implying that $T$ is not a contraction in the setting of modular space $V_{\rho^*}.$
\end{example}


\begin{thebibliography}{99}

\bibitem{Ref7}  Alsulami,~H. ~H., Agarwal, ~R. ~P., Karap\i nar,~E. and Khojasteh, ~F.,
\textit{A short note on $C^{\ast }$-algebra-valued contraction mappings}, Journal of
Inequalities and Applications {\bf 50}, (2016).

\bibitem{Ref4}
Huang, ~H., Radenovi\'{c}~S. and Deng, ~G., \textit{A sharp generalization on
cone $b$-metric space over Banach algebra}, J. Nonlinear Sci. Appl. {\bf 10}, (2017), 429--435.

\bibitem{Ref8} Kadelburg, ~ Z. and Radenovi\'{c}, ~S., \textit{Fixed point results in $C^{\ast
} $-algebra-valued metric spaces are direct consequences of their standard
metric counterparts}, Fixed point theory and applications {\bf 53}, (2016).

\bibitem{Ref3} Khamsi, ~M. ~A., Kozlowski, ~W. ~M. and  Reich, ~S., \textit{Fixed point theory in
modular function spaces}, Nonlinear Anal. {\bf 14}, (1990), no. 11,
935--953.

\bibitem{Ref18} Khamsi, ~M. ~A. and Kozlowski, ~W. ~M., \textit{Fixed point theory in
modular function spaces}, Springer/Birkh\"auser, New York (2015).

\bibitem{Ref2} 
Kozlowski, ~W. ~M., \textit{Modular Function Spaces}, Dekker, New York, 1988.

\bibitem{Ref14} Liu, ~H. and Xu, ~S., \textit{Cone metric spaces with Banach algebras and
fixed point theorems of generalized Lipschitz mappings}, Fixed Point Theory
and Applications {\bf 320}, (2013).

\bibitem{Ref6} Ma, ~Z., Jiang, ~L. and Sun, ~H., \textit{$C^{\ast }$-algebra-valued metric
spaces and related fixed point theorems}, Fixed Point Theory Appl. {\bf 206}, (2014), no. 1. 

\bibitem{Ref9} Murphy, ~G. ~J., \textit{$C^{\ast }$-algebra and operator theory}, Academic
press, INC. (1990).

\bibitem{Ref5} Musielak, ~J. \textit{Orlicz Spaces and Modular Spaces}, Lecture Notes in Mathematics. 1034, Springer, Berlin, (1983).

\bibitem{Ref13} Rudin, ~W., \textit{Functional Analysis} McGraw-Hill, New York, (1991).

\bibitem{Ref1} Shateri, ~T. ~L., \textit{$C^{\ast }$-algebra-valued modular spaces and
fixed point theorems}, J. Fixed Point Theory Appl. {\bf 19}, (2017), no. 2,
1551--1560.

\bibitem{Ref15} Xu, ~S. and Radenovi\'{c}, ~S., \textit{Fixed point theorems of generalized Lipschitz mappings on cone metric spaces over Banach algebras without assumption of normality}, Fixed Point Theory and Applications, {\bf 102}, (2014).

\bibitem{Ref19} Gerstewitz, Chr., \textit{Nichtkonvexe dualit\"at in der vektaroptimierung}, Wissenschaftlichte Zeitschrift T H Leuna-mersebung, {\bf 25}, (1983).


\bibitem{RefNA} Nakano, H., \textit{Modulared semi-ordered linear spaces}, Tokyo Mathematics Book Series, {\bf 1}, Maruzen Co., Tokyo (1950).

\bibitem{RefMUS} Musielak, J., Orlicz, W. \textit{On modular spaces}, Studia Math, {\bf 18}, 49-56 (1959).
\end{thebibliography}
\end{document}